\documentclass[a4paper,11pt]{amsart}
\usepackage{amsmath,amssymb}
\usepackage[latin1]{inputenc}
\usepackage{fancyhdr}
\usepackage{verbatim}

\usepackage{enumitem}
\setenumerate[1]{label=(\roman{*}), ref=(\roman{*})}
\setenumerate[2]{label=(\alph{*}), ref=(\alph{*})}

\newcommand\err{\mathbb R}

\newcommand{\eps}{\varepsilon}

\newcommand{\ncconv}{\overline{\mbox{conv}}}

\newcommand\Xs{X^\ast}
\newcommand\Xss{X^{\ast\ast}}
\newcommand\Xsss{X^{\ast\ast\ast}}

\newcommand{\HB}{\text{H{\kern -0.35em}B}}
\newcommand{\vertiii}[1]{{\left\vert\kern-0.25ex\left\vert\kern-0.25ex\left\vert#1\right\vert\kern-0.25ex\right\vert\kern-0.25ex\right\vert}}

\DeclareMathOperator{\spann}{span}

\DeclareMathOperator{\supp}{supp}
\newcounter{abcd}
{\setcounter{abcd}{0}
\begin{list}%
{{\rm (\alph{abcd})}} 
{\usecounter{abcd}
\parsep=\parskip
\topsep=1pt plus 2pt minus 1pt
\itemsep=1pt plus 2pt minus 1pt
\leftmargin=3\baselineskip \labelsep=.6\baselineskip
\labelwidth=2.4\baselineskip
\rightmargin 0pt}%
}%
{\end{list}}

\title{Diameter two properties, convexity and smoothness}

\author[T.~A.~Abrahamsen]{Trond A.~Abrahamsen}
\address[T.~A.~Abrahamsen]{Department of Mathematics, University of
  Agder, Postboks 422, 4604 Kristiansand, Norway.}
\email{trond.a.abrahamsen@uia.no}
\urladdr{http://home.uia.no/trondaa/index.php3}

\author[V.~Lima]{Vegard Lima}
\address[V.~Lima] {NTNU, Norwegian University of Science and
  Technology, Aalesund, Postboks 1517, N-6025 {\AA}lesund Norway.}
\email{Vegard.Lima@ntnu.no}

\author[O.~Nygaard]{Olav Nygaard}
\address[O.~Nygaard]{Department of Mathematics, University of Agder,
  Postboks 422, 4604 Kristiansand, Norway}
\email{Olav.Nygaard@uia.no}
\urladdr{http://home.hia.no/$~$olavn/}

\author[S.~Troyanski]{Stanimir Troyanski}
\address[S.~Troyanski]{
Institute of Mathematics and Informatics,
Bulgarian Academy of Science, bl.8,
acad. G. Bonchev str. 1113 Sofia, Bulgaria
and
Departamento de Matem{\'a}ticas, Universidad de Murcia,
Campus de Espinardo, 30100 Espinardo (Murcia), Spain
}
\email{stroya@um.es}

\thanks{The main results of this paper were presented in a talk by the
  fourth named author at the Seminario Matematico e Fisico di Milano
  in March 2016. The fourth named author was partially supported by
  MTM2014-54182-P and the Bulgarian National Scientific Fund under
  Grant DFNI-I02/10.}


\newtheorem{thm}{Theorem}[section]
\newtheorem{prop}[thm]{Proposition}
\newtheorem{lem}[thm]{Lemma}
\newtheorem{cor}[thm]{Corollary}

\theoremstyle{definition}

\newtheorem{defn}[thm]{Definition}

\theoremstyle{remark}

\newtheorem{rem}[thm]{Remark}

\begin{document}

\begin{abstract}
  We study smoothness and strict convexity of (the bidual) of Banach spaces
  in the presence of diameter 2 properties.
  We prove that the strong diameter 2 property prevents the bidual
  from being strictly convex and being smooth, and we initiate the investigation
  whether the same is true for the (local) diameter 2 property.
  We also give characterizations of the following property
  for a Banach space $X$: ``For every slice
  $S$ of $B_X$ and every norm-one element $x$ in $S$, there is a point
  $y\in S$ in distance as close to 2 as we want.''
  Spaces with this property are shown to have non-smooth bidual.
\end{abstract}

\maketitle

\section{Introduction}\label{sec:intro}

Let $X$ be a (real) Banach space and denote, as usual, by $B_X$ and
$S_X$ its unit ball and unit sphere, respectively, and denote the
topological dual of $X$ by $\Xs$.

Recall that (the norm of) a Banach space $X$ is \emph{strictly convex}
if $\|\frac{x+y}{2}\| < 1$ when $x$ and $y$ are different points of $S_X$,
and that (the norm of) $X$ is \emph{smooth} if for every
$x \in S_X$ there is exactly one $x^* \in S_{X^*}$ such that $x^*(x) = 1$.
It is well-known that $X$ is smooth if $\Xs$
is strictly convex, and that $X$ is strictly convex if $\Xs$ is smooth.

It is a classical result from 1948 of J.~Dixmier
\cite[Th\'{e}or\`{e}me~20']{Dix} that $X^{\ast\ast\ast\ast}$ is never
strictly convex unless $X$ is reflexive.
Several authors have independently strengthened Dixmier's result by
showing that $\Xsss$ is not smooth for $X$ non-reflexive.
(A partial list of authors can be found in \cite{Smith76}.
Milman credits the result to M.~I.~Kadets \cite[Theorem~2.3]{Mil71II}.)
Note that this result is sharp in the sense that
James' space $J$ has a renorming such that the third dual
is strictly convex \cite{Smith76}.

The purpose of this paper is to study
the implications of the big-slice phenomena on
smoothness and convexity.
By a slice of $B_X$ of $X$ we mean a
set of the form
\begin{equation*}
  S(x^*,\eps):=\{x \in B_X: x^*(x) > 1 - \eps, x^*\in S_{\Xs}, \eps>0\}.
\end{equation*}
Recall the following successively stronger ``big-slice concepts'',
defined in \cite{MR3098474}:

\begin{defn} \label{defn:diam2p} A Banach space $X$ has the
  \begin{enumerate}
  \item
    \emph{local diameter 2 property} (LD2P) if every slice of $B_X$ has
    diameter 2.
  \item
    \emph{diameter 2 property} (D2P) if every
    non-empty relatively weakly open subset of $B_X$ has diameter 2.
  \item
    \emph{strong diameter 2 property} (SD2P) if every finite convex
    combination of slices of $B_X$ has diameter 2.
  \end{enumerate}
\end{defn}

In Section~\ref{sec:bidualrotund} we prove that $\Xss$ can be neither
strictly convex nor smooth if $X$ has the SD2P. In fact, we prove that
when $X$ has the SD2P, then $\Xss$ contains an isometric copy of
$L_1[0,1]$. We next ask whether it is possible that $X$ can have
(L)D2P while $\Xss$ is still strictly convex, and we give a partial
answer; namely we prove that if $X$ has a bimonotone basis and the
D2P, then the unit sphere of $\Xss$ contains a line segment of length
as close to 1 as we want.

Recall the following successively stronger ``rotundity concepts'':
 
\begin{defn} \label{defn:rotundity} A Banach space $X$ is
\begin{enumerate}
  \item
    \emph{strictly convex} (or rotund) if every $x \in S_X$
    is an extreme point in $B_X$, i.e., for every
    $y \in X$ we have that $y=0$ whenever $\|x \pm y\| = 1$.
  \item
    \emph{weakly midpoint locally uniformly rotund} (weakly
    MLUR) if every $x\in S_X$ is a weakly strongly extreme point of
    $B_{X}$, i.e., for every sequence $(x_n)$ in $X$, we have that
    $x_n \to 0$ weakly whenever $\|x \pm x_n\| \to 1$.
  \item
    \emph{midpoint locally uniformly rotund} (MLUR) if every $x\in
    S_X$ is a strongly extreme point of $B_{X}$, i.e., for every
    sequence $(x_n)$ in $X$, we have that $x_n \to 0$ in norm whenever
    $\|x \pm x_n\| \to 1$.
  \end{enumerate}
\end{defn}

It is clear that if $X$ is weakly MLUR then $X$ is
strictly convex. Smith \cite{Smith84} observed using the
Principle of Local Reflexivity that $X$ is weakly MLUR if and only if every
$x\in S_X$ is an extreme point of $B_{\Xss}$. In particular, if $\Xss$
is strictly convex, then $X$ is weakly MLUR. The converse is not true.

It was observed in \cite[Proposition~1.3]{AHNTT}
that if $X$ is weakly MLUR, then the LD2P implies the D2P
by Choquet's lemma \cite[Lemma~3.69]{MR2766381}.
In particular, the LD2P implies the D2P when $\Xss$ is strictly convex.

The main result of \cite{AHNTT} is that there exists an equivalent
norm $|\cdot|$ on $C[0,1]$ such that $X=(C[0,1],|\cdot|)$ is MLUR and
has the (L)D2P. In fact $X$ has the LD2P in the following stronger
sense:

\begin{defn}\label{defn:LD2P+}
  A Banach space $X$ has the \emph{local diameter 2
    property+ (LD2P+)} if for every $\eps> 0$, every slice $S$ of
  $B_X$, and every $x \in S \cap S_X$ there exists $y \in S$ such that
  $\|x - y\| > 2 -\eps$.
\end{defn}

Let $X$ be a Banach space and $I$ the identity operator on $X$. Recall
that $X$ has the \emph{Daugavet property} if the equation
\begin{equation*}
   \|I+T\| =  1 + \|T\|
\end{equation*}
holds for every rank 1 operator $T$ on $X$. The Daugavet property
can be characterized as follows (see \cite{MR1856978} or \cite{MR1784413}):

\begin{thm}\label{thm:daugavet-char}
Let $X$ be a Banach space. Then the following statements are
equivalent.
\begin{enumerate}
  \item\label{item:daugchar-a}
    $X$ has the Daugavet property.
  \item\label{item:daugchar-b}
    The equation $\|I + T\|= 1 + \|T\|$ holds for every weakly
      compact operator $T$ on $X$.
  \item\label{item:daugchar-c}
    For every $\eps > 0$, every $x \in
   S_X$, and every $x^* \in S_{X^*}$, there exists $y \in S(x^*,\eps)$ such that
   $\|x + y\| \ge 2 - \eps$.
  \item\label{item:daugchar-d}
    For every $\eps > 0$, every $x^* \in
   S_{X^*}$, and every $x \in S_X$, there exists $y^* \in S(x,\eps)$ such that
   $\|x^* + y^*\| \ge 2 - \eps$.
  \item\label{item:daugchar-e}
    For every $\eps > 0$ and every $x \in S_X$ we have $B_X =
    \ncconv(\Delta_\eps(x))$, where $\Delta_\eps(x)=\{y \in B_X: \|y-x\|
    \ge 2-\eps\}$.
\end{enumerate}
\end{thm}
In Section~\ref{sec:ld2p+} we prove a similar characterization of the LD2P+,
see Theorems~\ref{thm:ikw} and \ref{prop:ld2p+-char}.
It is known that the dual of a Banach
space with the Daugavet property
is neither strictly convex nor smooth \cite[Corollary~2.13]{MR1621757}.
In Corollary~\ref{cor:ld2p+no-smooth-bidual} we show
that if $X$ has the LD2P+, then $\Xss$ is not smooth.
We also prove that just like the diameter two properties above
the LD2P+ is inherited by ai-ideals
(we postpone the definition of this concept till we need it).

The notation and conventions we use are standard and follow
\cite{JL}. When considered necessary, notation and concepts are
explained as the text proceeds.

\section{Strict convexity and smoothness of
  $\Xss$}\label{sec:bidualrotund}

A result of Day \cite[Theorem~9]{Day55} says that
neither $\ell_1(\Gamma)$, $\Gamma$ uncountable,
nor $\ell_\infty$ have equivalent smooth renormings.
So for example no equivalent norm on
$C[0,1]$ has a bidual that is smooth or strictly convex.

Our first aim in this section is to prove that
if a Banach space $X$ has the SD2P, then the bidual
is neither smooth nor strictly convex.
Banach spaces which are M-ideals in their biduals
are called \emph{M-embedded}.
It is known that non-reflexive M-embedded spaces
have the SD2P \cite[Theorem~4.10]{MR3098474}.
From \cite[p.~109]{MR547509}
(see also \cite[Proposition~I.1.7]{HWW}) it is clear
that the bidual of a non-reflexive M-embedded space
is neither smooth nor strictly convex.

Let us note that in general we cannot say anything
about the absence of smoothness or convexity in $X$ and $X^*$
in the presence of the SD2P. Indeed,
there exists a smooth M-embedded renorming $X$ of $c_0$ with
strictly convex dual \cite[Corollary~III.2.12]{HWW}.
There also exists a strictly convex M-embedded space $X$
with smooth dual \cite[Remark~IV.1.17]{HWW}.

We will need the following concept:

\begin{defn}
  A sequence $(x_n)$ in $X$ is said to be
  \emph{asymptotically isometric} $\ell_1$, if there
  exists a sequence $(\delta_n)$ in $(0,1)$ decreasing to $0$ and such
  that
  \begin{equation*}
     \sum_{n=1}^m(1 - \delta_n)|a_n| \le \|\sum_{n=1}^m a_n x_n\| \le
    \sum_{n=1}^m |a_n|
  \end{equation*}
   for each finite sequence $(a_n )_{n=1}^m$ in $\err$.
\end{defn}

From \cite[Remark~II.5.2]{G-octa} we have the following:

\begin{defn}
  A Banach space $X$ is said to
  be \emph{octahedral} if for every finite dimensional subspace $E$ of $X$
  and every $\eps > 0$, there exists $y \in S_X$ such that
  for every $x \in E$ and every $\lambda \in \err$, we have
  \begin{equation*}
    \| x + \lambda y \| \ge (1-\varepsilon)(\|x\|+|\lambda|).
  \end{equation*}
\end{defn}

\begin{lem}\label{lem:octa-asymp-ell1}
  If $X$ is octahedral, then $X$ contains an asymptotically isometric
  $\ell_1$ sequence.
\end{lem}

The proof uses an idea of H.~Pfitzner (see \cite[Theorem~2]{MR1814162}).

\begin{proof}
  Let $(\delta_n) \subset (0,1)$ such that $\delta_n \to 0$.
  Let $\eta_1 = \frac{\delta_1}{2}$
  and $\eta_{n+1} = \frac{1}{2}\min\{\eta_n, \delta_{n+1}\}.$ We will
  find a sequence $(x_n) \subset S_X$ such that
  \begin{align}
    \label{eq:3}
    \sum_{n=1}^m (1-\delta_n)|a_n| + \eta_m\sum_{n=1}^m|a_n| \le
    \|\sum_{n=1}^m a_nx_n\|
  \end{align}
  by induction. The $m=1$ step is trivial.

  Assume (\ref{eq:3}) holds for a fixed $m\ge 1$.
  Choose $\eps > 0$ such that
  \begin{align*}
    \eps \le \frac{\eta_m - \eta_{m+1}}{1-\delta_n + \eta_m}
  \end{align*}
  for $n=1,2,\ldots,m$
  and 
  \begin{align*}
    \eps \le \delta_{m+1} - \eta_{m+1}.
  \end{align*}
  Find, using the assumption
  and octahedrality, $x_{m+1} \in S_X$ such
  \begin{align}
    \label{eq:4}
    (1-\eps)\bigg(\sum_{n=1}^m (1-\delta_n)|a_n| +
    \eta_m\sum_{n=1}^m|a_n| + |a_{m+1}|\bigg) \le \|\sum_{n=1}^{m+1} a_nx_n\|.
  \end{align}
  Then 
  \begin{align}
    \label{eq:5}
    \sum_{n=1}^{m+1} (1-\delta_n)|a_n| + \eta_{m+1}\sum_{n=1}^{m+1}|a_n| \le
    \|\sum_{n=1}^{m+1} a_nx_n\|
  \end{align}
  because the left hand side of (\ref{eq:4}) in this case will be greater than the
  left hand side of (\ref{eq:5}).
\end{proof}

\begin{rem}
  As noted above there is a smooth M-embedded Banach space $X$
  with strictly convex dual and a strictly convex
  M-embedded Banach space $X$ with smooth dual.
  By \cite[Theorem~2]{MR1814162} $\Xs$ contains an
  asymptotically isometric $\ell_1$ sequence
  whenever $X$ is M-embedded.
  Hence the presence of an asymptotically isometric $\ell_1$ sequence
  in a Banach space $X$ does not prevent $X$ from
  being strictly convex or smooth -- even when $X$ is a dual space.
\end{rem}

\begin{thm}\label{thm:SD2P_case} If $X$ has the SD2P, then $\Xss$
  contains an isometric copy of $L_1[0,1]$.
\end{thm}

\begin{proof}
  We know from \cite[Theorem~2.4]{HLP} that $X$ has the SD2P if and
  only if $X^*$ is octahedral.
  From Lemma~\ref{lem:octa-asymp-ell1} we know that an octahedral space
  contains an asymptotically isometric $\ell_1$ sequence. From
  \cite[Theorem~2]{MR1751149} we then have that $X^{**}$ contains
  an isometric copy of $L_1[0,1]$.
\end{proof}

Since $L_1[0,1]$ is neither smooth nor strictly convex
the following corollary is immediate.

\begin{cor}\label{cor:SD2P-bidual-nsc}
  If $X$ has the SD2P, then $\Xss$ is neither strictly convex nor
  smooth.
\end{cor}

From Corollary~\ref{cor:SD2P-bidual-nsc} a natural question arises:
If $X$ has the D2P, can $\Xss$ be strictly convex?
(Recall from the Introduction that when $\Xss$ is strictly convex,
LD2P and D2P for $X$ must be the same thing.)
We will give a negative answer to this question in the case
$X$ has a bimonotone basis in Proposition~\ref{prop:D2Pcase_bsmono} below.

We start with an alternative description of the D2P.

\begin{prop}\label{alterd2} The following statements are equivalent:
\begin{enumerate}
  \item $X$ has the D2P.
  \item Whenever $\eps>0$, $x \in X$ with $\|x\|<1$, and $F$ is a
    finite dimensional subspace of $X^\ast$, there exist $y_1, y_2\in
    F_{\perp}$ with $\|x+y_i\|<1$, $i=1,2$, such that $\|y_1-y_2\|>2-\eps$.
  \item Whenever $\eps>0$, $x \in X$ with $\|x\|<1$, and $E$ is a finite
    co-dimensional subspace of $X$, there exists $y_1, y_2\in E$ with
    $\|x+y_i\|<1$, $i=1,2$, such that $\|y_1-y_2\|>2-\eps$.
\end{enumerate}
\end{prop}

\begin{proof}
  (i) $\Rightarrow$ (iii). Let $\eps > 0$, $x \in X$
  with $\|x\| < 1$, and $E$ a finite co-dimensional subspace of
  $X$. Assume without
  loss of generality that $E$ does not contain $x$.
  Choose a finite dimensional subspace $F$ of $X$ which
  contains $x$ and with the property that $X = E \oplus F$.
  Let $P$ be a bounded  linear projection onto $F$.
  For $\eps/5>\delta>0$  put
  \begin{equation*}
    W = \{ w \in B_X \,: \|P(x-w)|| < \delta \}.
  \end{equation*}
  Note that
  $W$ is a neighbourhood of $x$ in the relative weak topology on
  $B_X$. Now, using (i), and that non-empty
  relatively weakly open subsets of $B_X$ has diameter 2, we may pick
  $w_1, w_2$ in $W$, both of norm $<1-\delta$ and with
  $\|w_2 - w_1\| > 2 - 3 \delta$.

  Put $y_i=w_i-Pw_i$.
  Then $y_1$ and $y_2$ are both in $E$.
  Moreover, for $i=1,2$, we have
  \begin{equation*}
    \|x + y_i\| = \|Px + w_i - Pw_i\| \le \|P(x-w_i)\| + \|w_i\| < 1.
  \end{equation*}
  We also have
  \begin{align*}
    \|y_1-y_2\|&=\|w_1-w_2 - P(w_1-w_2)\|\\
               &\geq
                 \|w_1-w_2\|-2\delta
                 >2-5\delta>2-\eps
  \end{align*}
  since $\|P(w_1-w_2)\| < 2\delta$.

 (iii) $\Rightarrow$ (ii). This is obvious as any finite dimensional
 subspace of a dual space has a finite co-dimensional pre-annihilator.

 (ii) $\Rightarrow$ (i). Let $\eps>0$ and $U$ a non-empty relatively
 weakly open subset of $B_X$.
 Let $x \in U$ with $\|x\| <1$
 and find a set of the form
 \begin{equation*}
   V = \left(
     x + \bigcap_{k=1}^{n}
     (f_k)^{-1}(-\delta,\delta)
   \right)
   \bigcap B_X \subset U,
 \end{equation*}
 where
 $(f_k)_{k=1}^{n} \subset S_{X^*}$ and $\delta > 0$.
 Let
 \begin{equation*}
   F =
   \spann \{ (f_k)_{k=1}^{n} \}.
 \end{equation*}
 As $F$ is of finite
 dimension in $X^*$, there exist $y_1, y_2\in F_\perp$ with
 $\|x + y_i\|<1$, $i=1,2$, such that
 $\|y_2-y_1\| > 2-\eps$.
 For $i=1,2$ we have $x + y_i\in V$
 with $\|(x+y_1)-(x+y_2)\| > 2 - \eps$.
\end{proof}

As a first application of Proposition~\ref{alterd2} let us give a very
simple proof of the following fact, known from \cite{BGLPRZ6}.

\begin{prop} If $X$ has the D2P and $Y$ is a subspace of $X$ with finite co-dimension, then $Y$ has the D2P.
\end{prop}

\begin{proof}
  If $y\in Y$ with $\|y\|<1$, $\eps>0$, and $E$ is of finite co-dimension
  in $Y$, then $E$ is also of finite
  co-dimension in $X$ and the result follows from
  Proposition~\ref{alterd2}~(iii).
\end{proof}

Now we return to the problem whether $\Xss$ can be strictly convex if
$X$ has the D2P.
\begin{defn}
  A Schauder basis $(e_k)_{k=1}^\infty$ for a Banach space $X$
  is \emph{bimonotone} if the projections
  \begin{equation*}
    P_{[n,m]}(\sum_{k=1}^\infty a_k e_k) = \sum_{k=n}^m a_k e_k.
  \end{equation*}
  satisfy $\|P_{[n,m]}\| = 1$ if $n \le m$.
\end{defn}

\begin{prop}\label{prop:D2Pcase_bsmono}
  Suppose $X$ has a bimonotone basis. Then, if $X$ has the
  D2P and $\varepsilon>0$, $S_{\Xss}$ contains a line segment of
  length $>1-\varepsilon$.
\end{prop}

\begin{proof}
  Let $P_n$ be the natural projections associated to the basis $(e_i)$
  and put $Q_n=I-P_n$.
  Let $\eps > 0$ and define $\eps_n = \eps/2^{n+1}$.

  Define $s_1 = x_1 = \frac{(1-\varepsilon_1)e_1}{\|e_1\|}$.
  Assume that we have found a sequence $(x_i)_{i=1}^k$
  each with finite support $\supp(x_i) = [l_i,r_i]$
  such that if $s_k=\sum_{i=1}^k x_i$, then
  \begin{itemize}
  \item
    $\|s_k\|<1$ and $\|s_k\|\geq \|s_{k-1}\|$
  \item
    $\|x_i\| > 1-\varepsilon_i$ for $i=1,2,\ldots,k$.
  \item
    $r_i < l_{i+1}$ for $i=1,2,\ldots,k-1$.
  \end{itemize}

  Let us show how to find $x_{k+1}$.
  Let $E = Q_{r_k}(X)$ and use Proposition~\ref{alterd2} to
  find $y_1,y_2 \in E$ with $\|s_k + y_i\| < 1$ and
  $\|y_1-y_2\| > 2 - 2\varepsilon_{k+1}$.
  Without loss of generality
  $\|y_1\| \ge \|y_2\|$ and $y_1$ has finite support.
  Let $x_{k+1} = y_1$.
  Then $\|s_{k+1}\| = \|s_k + x_{k+1}\| < 1$
  and $\|s_k\| \le \|s_{k+1}\|$ since the basis is monotone.
  We also have
  \begin{equation*}
    2\|x_{k+1}\| = 2 \|y_1\| \ge \|y_1 - y_2\| > 2 - 2 \varepsilon_{k+1}
  \end{equation*}
  so $\|x_{k+1}\| > 1-\eps_{k+1}$.

  Let $\mathcal{U}$ be a non-trivial ultrafilter on $\mathbb{N}$.
  Then $z = w^\ast-\lim_{\mathcal{U}} s_m \in \Xss$ exists with
  $\|z\| \le 1$.
  For $\lambda \in [0,1]$, let
  \begin{equation*}
    z_\lambda = w^\ast-\lim_{\mathcal{U}} (s_m-\lambda s_1) =
    z - \lambda s_1.
  \end{equation*}
  We have $z_\lambda = \lambda z_1 + (1-\lambda) z_0$.
  Note that $\|s_m - s_1\| \le \|s_m\| < 1$
  since the basis is bimonotone.
  Hence $\|z_1\| \le 1$ and $\|z_0\| \le 1$,
  so the line segment $[z_0,z_1]$ is contained in $B_{\Xss}$.
  Let $R_n = P_{[l_n,r_n]}$ be the projection onto the support of
  $x_n$. We have $\|R_n\| = 1$ and
  \begin{equation*}
    R_n^{**} z_\lambda =
    w^\ast-\lim_{\mathcal{U}} R_n(s_m - \lambda s_1) = x_n
  \end{equation*}
  and hence $\|z_\lambda\| \ge \|R_n^{**} z_\lambda\| = \|x_n\| > 1-\eps_n$
  for all $n$ which means that $\|z_\lambda\| = 1$.
  Thus $z_\lambda = \lambda z_1 + (1-\lambda) z_0$, $\lambda \in
  [0,1]$, is a line segment on the sphere.
  The segment has length $\|z_0 - z_1\| = \|s_1\| > 1 - \varepsilon$.
\end{proof}

\section{The local diameter 2 property+}\label{sec:ld2p+}

Let us recall from the Introduction the definition of the
LD2P+.

\begin{defn}
  We say that a Banach space $X$ has the \emph{local diameter 2
    property+} (LD2P+) if for every $x^* \in S_{X^*}$, every
  $\eps > 0$, every $\delta > 0$, and every $x \in
  S(x^*,\eps) \cap S_X$ there exists $y \in S(x^*,\eps)$ with $\|x
  -y\| > 2-\delta$.
\end{defn}

From \cite[Theorem~1.4]{IK} and \cite[Open problem (7)
p.~95]{MR1856978} the following is known.

\begin{thm}\label{thm:ikw}
  Let $X$ be a Banach space. Then the following statements are
  equivalent.
  \begin{enumerate}
    \item\label{item:ikw-a}
      The equation $\|I - P\| = 2$ holds for every norm 1 rank 1
      projection $P$ on $X$.
    \item\label{item:ikw-b}
      For every $\eps>0$, every $x^* \in S_{X^*}$ and every
      $x \in S(x^*,\eps) \cap S_X$ there exists $y \in S(x^*,\eps)$
      with $\|x-y\| > 2-\eps.$
    \item\label{item:ikw-c}
      For every $x \in S_X$ and every $\eps >0$ we have $x
      \in \ncconv(\Delta_\eps(x))$, where $\Delta_\eps(x) = \{y \in
      B_X: \|x - y\| > 2 - \eps\}.$
  \end{enumerate}
\end{thm}

From Lemma \ref{lem:subslice} below, which is due to Ivakhno and Kadets  \cite[Lemma~2.1]{IK}, it is clear that the LD2P+ is
equivalent to the statements in Theorem \ref{thm:ikw}. Therefore
every Daugavet space has the LD2P+. Note, however, that the converse
is not true as the LD2P+ is stable by taking unconditional sums of
Banach spaces which fails for spaces with the Daugavet property (see
e.g. \cite[Corollary~3.1]{IK}).

\begin{lem}[Ivakhno and Kadets]\label{lem:subslice}
  Let $\eps>0$ and $x^* \in S_{X^*}$. Then for every $x \in S(x^*,\eps) \cap
  S_X$ and every positive $\delta < \eps$
  there exist $y^* \in S_{X^*}$ such that  $x \in S(y^*,\delta)$ and
  $S(y^*,\delta) \subset S(x^*,\eps).$
\end{lem}

In the proof of Theorem~\ref{prop:ld2p+-char} below we will need
the following weak$^*$-version of Lemma \ref{lem:subslice}. Its proof
is more or less verbatim to that of Lemma \ref{lem:subslice} and will
therefore be omitted.

\begin{lem}\label{lem:w*subslice}
  Let $\eps>0$ and $x \in S_X$. Then for every $x^* \in S(x,\eps) \cap
  S_{X^*}$ which attains its norm and every positive $\delta < \eps$
  there exist $y \in S_X$ such that  $x^* \in S(y,\delta)$ and
  $S(y,\delta) \subset S(x,\eps).$
\end{lem}

We will now add to the list of statements in Theorem \ref{thm:ikw}
statements similar to \ref{item:daugchar-b} and \ref{item:daugchar-d}
in Theorem \ref{thm:daugavet-char}. As pointed out in
\cite[p.~232]{AHNTT} the equivalence of
\ref{item:l2+-a} and \ref{item:l2+-b} in Theorem~\ref{prop:ld2p+-char}
below can be proved by a similar argument to the proof of
\cite[Lemma~1.5]{MR1621757}.

\begin{thm}\label{prop:ld2p+-char}
  Let $X$ be a Banach space. Then the following statements are equivalent:
  \begin{enumerate}
    \item\label{item:l2+-a}
      $X$ has the LD2P+.
    \item\label{item:l2+-b}
      For every $x \in S_X$, every $\eps > 0$, every
      $\delta > 0$, and every $x^* \in S(x,\eps) \cap S_{X^*}$ there
      exists $y^* \in S(x,\eps)$ with $\|x^* -y^*\| > 2-\delta$.
    \item\label{item:l2+-c}
      The equation $\|I - P\|= 1 + \|P\|$ holds for every weakly
      compact projection $P$ on $X$.
  \end{enumerate}
\end{thm}

\begin{proof}
  \ref{item:l2+-a} $\Rightarrow$ \ref{item:l2+-b}. By the Bishop-Phelps
  theorem we can assume without loss of generality that $x^* \in
  S(x,\eps) \cap S_{X^*}$ attains its norm. Let $0 < \eta <
  \min\{\eps,\delta/2\}$ and find by Lemma  \ref{lem:w*subslice} $y \in
  S_X$ such that $x^* \in S(y,\eta)$ and $S(y,\eta) \subset
  S(x,\eps)$. Note that $y \in S(x^*,\eta)$ and thus,  since $X$ has
  the LD2P+, we can find $z \in S(x^*,\eta)$ such that  $\|y - z\| >
  2-\eta.$ Hence there is $y^* \in S_{X^*}$ such that \[y(y^*)-z(y^*)
  = (y-z)(y^*) > 2 - \eta.\] From this we have $y(y^*) > 1-\eta$ and
  $-z(y^*) > 1 - \eta$. It follows that $y^* \in S(x,\eps)$ as $S(y,\eta) \subset
 S(x,\eps)$. Moreover, using that $z \in S(x^*,\eta)$, we have
 \begin{align*}
   \|x^* - y^*\| &\ge (x^*-y^*)(z) \\
                 &= x^*(z) - y^*(z) \\
                 & > 1 -  \eta + 1 - \eta > 2 - \delta.
 \end{align*}

 \ref{item:l2+-b} $\Rightarrow$ \ref{item:l2+-a}. The proof is
 identical to the proof of the converse except
 that one does not use the Bishop-Phelps theorem and that
 one uses Lemma \ref{lem:subslice} in place of Lemma
 \ref{lem:w*subslice}.

 \ref{item:l2+-a} $\Rightarrow$ \ref{item:l2+-c}.
 The proof is similar to that of
 \cite[Theorem~2.3]{MR1621757}.

 \ref{item:l2+-c} $\Rightarrow$ \ref{item:l2+-a}.
 This is clear as \ref{item:l2+-c} trivially implies \ref{item:ikw-a}
 in Theorem \ref{thm:ikw}.
\end{proof}

Note that from Theorem \ref{prop:ld2p+-char} we get

\begin{cor}\label{cor:m-emb-ld2p+}
  If $B_{X^{\ast}}$ contains a weak$^\ast$-denting point, in particular if $X^\ast$ has the RNP, then $X$ does not have the LD2P+.
\end{cor}

It is known that if $\Xss$ is smooth, then $X^*$ has
the RNP (see e.g. \cite{Sul77}), hence we have the following corollary.

\begin{cor}\label{cor:ld2p+no-smooth-bidual}
  If $X$ has the LD2P+, then $\Xss$ is not smooth.
\end{cor}

It is known that all the diameter 2 properties in Definition \ref{defn:diam2p} as
well as the Daugavet
property are inherited by certain subspaces called ai-ideals (see
\cite{ALN2} and \cite{A3}). We will end this section by showing that
this is true for the LD2P+ as well.

A subspace $X$ of a Banach space $Y$ is called
an \emph{ideal} in $Y$ if there exists a norm 1 projection $P$ on $Y^*$  with
$\ker P = X^\perp $. $X$ being an ideal in $Y$ is in turn equivalent
to $X$ being locally 1-complemented in $Y$, i.e., for every $\eps>0$ and
every finite dimensional subspace $E\subset Y$ there exists a linear $T:E\to X$
such that
\begin{itemize}
  \item [(i)] $Te=e$ for all $e\in X\cap E$.
  \item [(ii)] $\|Te\|\leq (1+\eps)\|e\|$ for all $e\in E$.
\end{itemize}

Following \cite{ALN2} a subspace $X$ of a Banach space $Y$ is called
an \emph{almost isometric ideal (ai-ideal)} in $Y$ if $X$ is locally 1-complemented
with almost isometric local projections, i.e., for every $\eps>0$ and
every finite-dimensional subspace $E\subset Y$ there exists $T:E\to X$
which satisfies a) and
\begin{itemize}
  \item [(ii')] $(1-\eps)\|e\|\leq\|Te\|\leq (1+\eps)\|e\|$ for all $e\in E$.
\end{itemize}

Note that an ideal $X$ in $Y$ is an ai-ideal if $P(Y^*)$ is a 1-norming subspace of
$Y^\ast$ \cite[Proposition~2.1]{ALN2}. Ideals $X$ in $Y$ for which $P(Y^*)$ is a 1-norming
subspace for $Y$ are called \emph{strict ideals}. An ai-ideal is, however,
not necessarily strict (see \cite[Example~1]{ALN2} and \cite[Remark~3.2]{ALLN}).

\begin{prop} \label{prop:ld2p+-ai-ideal}
  Let $Y$ have the LD2P+ and assume $X$ is an ai-ideal in $Y$. Then $X$
  has the LD2P+.
\end{prop}

\begin{proof} For $\delta>0$, $Z$ a subspace of $Y$, and $x \in S_Z$ put
  \[\Delta_\delta^Z(x) = \{y \in B_Z: \|x-y\| > 2- \delta\}.\]

  Let $x\in S_X$, $\eps >0$, and $\alpha >0$. We will show that there
  exists $z \in \mbox{conv}\Delta_\eps^X(x)$ with $\|x-z\| <
  \alpha$. The result will then follow from
  Theorem~\ref{thm:ikw}~\ref{item:ikw-c}. First, since $Y$
  enjoys the LD2P+, we know that for any positive $\beta < \eps$ and
  any positive $\gamma < \alpha$ we can find
  $y=\sum_{n=1}^N\lambda_ny_n$ with
  $(y_n)_{n=1}^N \subset \Delta_{\beta}^Y(x)$
  such that $\|x-y\| < \gamma$. Now
  let $E=\spann\{y_1, \ldots, y_N, x\}$ and pick a
  local projection $T:E\to X$ such that $T$ is a
  $(1+\eta)$-isometry with $\eta > 0$ so small that
  $(1+\eta)\gamma + \eta <  \alpha$, and $(1-\eta)(2-\beta) - \eta > 2 -
  \eps$. Put $z_n=\frac{Ty_n}{\|Ty_n\|}$ and
  $z=\sum_{n=1}^N\lambda_nz_n$. As $Tx=x$ we get
  \begin{align*}
    \|x-z\| &\le \|x-Ty\|+ \|Ty - z\|\\
                & \le \|T(x -y)\| + \sum_{n=1}^N\lambda_n\big|1 - \|Ty_n\|\big| \\
                &< (1+ \eta)\gamma + \max_{1 \le n \le N}{\big|1 -
                  \|Ty_n\|\big|}\\
                & \le (1+ \eta)\gamma  + \eta < \alpha.
  \end{align*}

  Moreover, for every $1 \le n \le N$ we have,
  \begin{align*}
    \|x - z_n\| & = \|T(x - \frac{y_n}{\|Ty_n\|})\| \\ & \ge (1-\eta)\|x -
                        \frac{y_n}{\|Ty_n\|}\|\\
                      & \ge (1-\eta)(\|x - y_n\| - \|y_n -
                        \frac{y_n}{\|Ty_n\|}\|)\\
                       & \ge (1-\eta)(2-\beta -
                         \frac{\big|1-\|Ty_n\|\big|}{\|Ty_n\|}\|y_n\|)\\
                       & \ge (1-\eta)(2-\beta -
                         \frac{\eta}{1-\eta}) > 2- \eps,
  \end{align*}
  Thus $(z_n)_{n=1}^N \subset \Delta_\eps(x)$ and as
  $\alpha > 0$ is arbitrarily chosen, we are done.
\end{proof}

\def\cprime{$'$} \def\cprime{$'$}
\providecommand{\bysame}{\leavevmode\hbox to3em{\hrulefill}\thinspace}
\providecommand{\MR}{\relax\ifhmode\unskip\space\fi MR }
\providecommand{\MRhref}[2]{%
  \href{http://www.ams.org/mathscinet-getitem?mr=#1}{#2}
}
\providecommand{\href}[2]{#2}

\end{document}